\documentclass{llncs}

\usepackage{microtype,comment}
\usepackage{xypic}

\title{Disjoint Paths and Connected Subgraphs\\ for $H$-Free Graphs}

\author{Walter Kern\inst{1}\thanks{Walter Kern recently passed away and we are grateful for his contribution.} \and Barnaby Martin\inst{2}\and Dani\"el Paulusma\inst{2}\thanks{Dani\"el Paulusma was supported by the Leverhulme Trust (RPG-2016- 258).}  \and\\ Siani Smith\inst{2} \and Erik Jan van Leeuwen\inst{3}}

\institute{Department of Applied Mathematics, University of Twente, The Netherlands,
\email{w.kern@twente.nl}
\and
Department of Computer Science,
 Durham University, Durham, UK,
\email{\{barnaby.d.martin,daniel.paulusma,siani.smith\}@durham.ac.uk}
\and
Department of Information and Computing Sciences, 
Utrecht University,\\
The Netherlands,
\email{e.j.vanleeuwen@uu.nl}
}

\newcommand{\ssi}{\subseteq_i}

\usepackage{enumerate}
\usepackage{tikz}
\usepackage{boxedminipage,amsmath}
\usetikzlibrary{arrows,shapes,calc}

\newcommand{\problemdef}[3]{
	\begin{center}
		\begin{boxedminipage}{1.02\textwidth}
			\textsc{{#1}}\\[1pt]  
			\begin{tabular}{ r p{0.8\textwidth}}
				\textit{~~~~Instance:} & {#2}\\
				\textit{Question:} & {#3}
			\end{tabular}
		\end{boxedminipage}
	\end{center}
}

\newcommand{\NP}{{\sf NP}}

\newcounter{ctrclaim}[theorem]
\newcounter{ctrcase}[theorem]

\begin{document}
\maketitle

\begin{abstract}
The well-known {\sc Disjoint Paths} problem is to decide if a graph contains $k$ pairwise disjoint paths, each connecting a different terminal pair from a set of $k$ distinct pairs.
We determine, with an exception of two cases, the complexity of the {\sc Disjoint Paths} problem for $H$-free graphs. 
If $k$ is fixed, we obtain the {\sc $k$-Disjoint Paths} problem, which is known to be polynomial-time solvable on the class of all graphs for every $k\geq 1$.
The latter does no longer hold if we need to connect vertices from terminal sets instead of 
terminal
pairs. We completely classify the complexity of {\sc $k$-Disjoint Connected Subgraphs} for $H$-free graphs, and give the same almost-complete classification for {\sc Disjoint Connected Subgraphs} for $H$-free graphs as for {\sc Disjoint Paths}.

\end{abstract}

\section{Introduction}

A path from $s$ to $t$ in a graph~$G$ is an {\it $s$-$t$-path} of $G$, and $s$ and $t$ are called its {\it terminals}.  Two pairs $(s_1,t_1)$ and $(s_2,t_2)$ are {\it disjoint} if $\{s_1,t_1\}\cap \{s_2,t_2\}=\emptyset$. In 1980, Shiloach~\cite{Sh80} gave a polynomial-time algorithm for testing if a graph with disjoint terminal pairs $(s_1,t_1)$ and $(s_2,t_2)$ has vertex-disjoint paths $P^1$ and $P^2$ such that each $P^i$ is an $s_i$-$t_i$ path. This problem can be generalized as follows.

\problemdef{Disjoint Paths}{a graph $G$ and pairwise~disjoint~terminal~pairs~$(s_1,t_1)\ldots,(s_k,t_k)$.}{Does $G$ have pairwise vertex-disjoint paths $P^1$,\ldots,$P^k$ such that $P^i$ is an $s_i$-$t_i$ path for $i\in \{1,\ldots,k\}$?}

\noindent
Karp~\cite{Ka75} proved that {\sc Disjoint Paths} is \NP-complete. If $k$ is fixed, that is, not part of the input, then we denote the problem as {\sc $k$-Disjoint Paths}. For every $k\geq 1$, Robertson and Seymour proved the following celebrated result.
\begin{theorem}[\cite{RS95}]\label{t-known}
For all $k\geq 2$, {\sc $k$-Disjoint Paths} is polynomial-time
solvable.
\end{theorem}

\noindent
The running time in Theorem~\ref{t-known} is cubic. This was later improved to quadratic time by Kawarabayashi, Kobayashi and Reed~\cite{KKR12}.

As {\sc Disjoint Paths} is \NP-complete, it is natural to consider special graph classes. The {\sc Disjoint Paths} problem is known to be \NP-complete even for graph of clique-width at most~$6$~\cite{GW06}, split graphs~\cite{HHLS15}, interval graphs~\cite{NS96} and line graphs. The latter result can be obtained by a straightforward reduction (see, for example,~\cite{GW06,HHLS15}) from its edge variant, {\sc Edge Disjoint Paths}, proven to be \NP-complete by Even, Itai and Shamir~\cite{EIS76}. On the positive side, {\sc Disjoint Paths} is polynomial-time solvable for cographs, or equivalently, $P_4$-free graphs~\cite{GW06}. 

We can generalize the {\sc Disjoint Paths} problem by considering terminal sets~$Z_i$ instead of terminal pairs $(s_i,t_i)$. We write $G[S]$ for the subgraph of a graph $G=(V,E)$ induced by $S\subseteq V$, where $S$ is {\it connected} if $G[S]$ is connected.

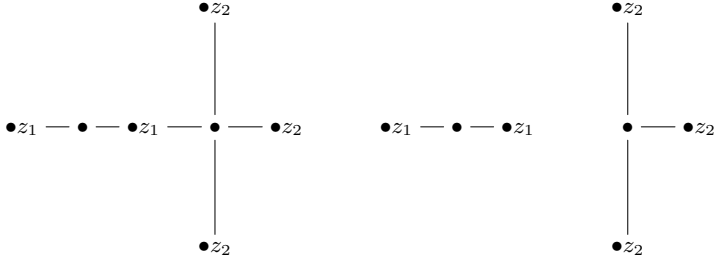
\begin{figure}[t]
		\resizebox{5.5cm}{!} {
			\begin{minipage}{0.45\textwidth}
				\centering
				\[
				\xymatrix@R3pc @C0.75pc{ 
				&&&\bullet{z_2} \ar@{-}[d]&&&&&&& \bullet{z_2} \ar@{-}[d]\\
				\bullet{z_1} \ar@{-}[r] & \bullet \ar@{-}[r] & \bullet{z_1} \ar@{-}[r] & \bullet \ar@{-}[r] & \bullet{z_2} && \bullet{z_1} \ar@{-}[r] & \bullet \ar@{-}[r] & \bullet{z_1} && \bullet \ar@{-}[r] & \bullet{z_2}  \\ 
				&&& \bullet{z_2} \ar@{-}[u] &&&&&&& \bullet{z_2} \ar@{-}[u]
				}
				\]
			\end{minipage}
		
		}
		\caption{An example of a yes-instance $(G,Z_1,Z_2)$ of {\sc ($2$-)Disjoint Connected Subgraphs} (left) together with a solution (right).}\label{f-dcs}
\end{figure}

\problemdef{Disjoint Connected Subgraphs}{a graph $G$ and pairwise~disjoint~terminal~sets~$Z_1,\ldots,Z_k$.}{Does $G$ have pairwise disjoint connected sets $S_1,\ldots, S_k$ such that $Z_i\subseteq S_i$ for $i\in \{1,\ldots,k\}$?}

\noindent
If $k$ is fixed, then we write $k$-{\sc Disjoint Connected Subgraphs}. 
We refer to Figure~\ref{f-dcs} for a simple example of an instance $(G,Z_1,Z_2)$ of $2$-{\sc Disjoint Connected Subgraphs}.
Robertson and Seymour~\cite{RS95} proved in fact that $k$-{\sc Disjoint Connected Subgraphs} is cubic-time solvable as long as $|Z_1|+\ldots +|Z_k|$ is fixed (this result implies Theorem~\ref{t-known}). 
Otherwise, van 't Hof et al.~\cite{HPW09} proved that  already $2$-{\sc Disjoint Connected Subgraphs} is \NP-complete even if $|Z_1|=2$ (and $|Z_2|$ may have arbitrarily large size). The same authors also proved that $2$-{\sc Disjoint Connected Subgraphs} is \NP-complete for split graphs. 
Afterwards, Gray et al.~\cite{GKLS12} proved that {\sc $2$-Disjoint Connected Subgraphs} is \NP-complete for planar graphs.
Hence, Theorem~\ref{t-known} cannot be extended to hold for {\sc $k$-Disjoint Connected Subgraphs}.

We note that in recent years a number of exact algorithms were designed for {\sc $k$-Disjoint Connected Subgraphs}. 
Cygan et al.~\cite{CPPW14} gave an $O^*(1.933^n)$-time algorithm for the case $k=2$ (see~\cite{PR11,HPW09} for faster exact algorithms for special graph classes).
Telle and Villanger~\cite{TV13} improved this to time $O^*(1.7804^n)$.
Recently, Agrawal et al.~\cite{AFLST20} gave an  $O^*(1.88^n)$-time algorithm for the case $k=3$.  
Moreover, the {\sc $2$-Disjoint Connected Subgraphs} problem plays a crucial role in graph contractibility: a connected graph can be contracted to the $4$-vertex path if and only if there exist two vertices $u$ and $v$ such that $(G-\{u,v\},N(u),N(v))$ is a yes-instance of {\sc $2$-Disjoint Connected Subgraphs} (see, e.g.~\cite{KP20,HPW09}).

A class of graphs that is closed under vertex deletion is called {\it hereditary}. Such a graph class can be characterized by a unique set~${\cal F}$ of minimal forbidden induced subgraphs. Hereditary graphs enable a systematic study of the complexity of a graph problem under input restrictions: by starting with the case where $|{\cal F}|=1$, we may already obtain more general methodology and a better understanding of the complexity of the problem. 
If $|{\cal F}| = 1$, say ${\cal F} = \{H\}$ for some graph $H$, then we obtain the class of {\it $H$-free} graphs, that is, the class of graphs that do not contain $H$ as an induced subgraph (so, an $H$-free graph cannot be modified to $H$ by vertex deletions only).
In this paper, we start such a systematic study for {\sc Disjoint Paths} and {\sc Disjoint Connected Subgraphs}, both for the case when $k$ is part of the input and when $k$ is fixed.

\subsection*{Our Results}

By combining some of the aforementioned known results with a number of new results, we prove the following two theorems in Sections~\ref{s-main1} and~\ref{s-main2}, respectively. In particular, we generalize the polynomial-time result for {\sc Disjoint Paths} on $P_4$-free graphs to hold even for {\sc Disjoint Connected Subgraphs}. See Figure~\ref{f-3p1p4} for an example of a graph $H=sP_1+P_4$; we refer to Section~\ref{s-pre} for undefined terminology.

\begin{theorem}\label{t-main1}
Let $H$ be a graph. If $H\ssi sP_1+P_4$, then for every  $k\geq 2$, {\sc $k$-Disjoint Connected Subgraphs} on $H$-free graphs is polynomial-time solvable; 
otherwise even {\sc $2$-Disjoint Connected Subgraphs} is \NP-complete.
\end{theorem}

\begin{figure}[t]
		\resizebox{6.5cm}{!} {
			\begin{minipage}{0.45\textwidth}
				
				\[
				\xymatrix@R3pc @C0.75pc{ 
					\bullet & \bullet & \bullet & \bullet \ar@{-}[r] & \bullet \ar@{-}[r] & \bullet \ar@{-}[r] & \bullet \\
				}
				\]
			\end{minipage}
		
		}
		\caption{The graph $H=3P_1+P_4$.}\label{f-3p1p4}
\end{figure}
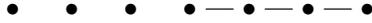

\begin{theorem}\label{t-main2}
Let $H$ be a graph not in $\{3P_1,2P_1+P_2,P_1+P_3\}$. If $H\ssi P_4$, then {\sc Disjoint Connected Subgraphs} 
is polynomial-time solvable for $H$-free graphs; otherwise even {\sc Disjoint Paths} 
is \NP-complete.
\end{theorem}

\noindent
Theorem~\ref{t-main1} completely classifies, for every $k\geq 2$, the complexity of {\sc $k$-Disjoint Connected Subgraphs} on $H$-free graphs. Theorem~\ref{t-main2} determines the complexity of {\sc Disjoint Paths} and {\sc Disjoint Connected Subgraphs} on $H$-free graphs for every graph $H$ except if $H\in \{3P_1,2P_1+P_2,P_1+P_3\}$. In Section~\ref{s-reduce} we reduce the number of open cases from six to {\it three} by showing some equivalencies. 

In Section~\ref{s-exact} we complement the above results by giving exact algorithms for both problems based on Held-Karp type dynamic programming techniques~\cite{HeldKarp,Bellman}.
In Section~\ref{s-con} we give some directions for future work. In particular we prove that both problems are polynomial-time solvable for co-bipartite graphs, which form a subclass of the class of $3P_1$-free graphs.

\section{Preliminaries}\label{s-pre}

We use $H \ssi H'$ to indicate that $H$ is an induced subgraph of $H'$, that is, $H$ can be obtained from $H'$ by a sequence of vertex deletions.
For two graphs $G_1$ and $G_2$ we write $G_1+G_2$ for the {\it disjoint union} $(V(G_1)\cup V(G_2),E(G_1)\cup E(G_2))$. We denote the disjoint union of $r$ copies of a graph $G$ by $rG$. A graph is said to be a linear forest if it is a disjoint union of paths.

We denote the path and cycle on $n$ vertices by $P_n$ and $C_n$, respectively.
The {\it girth} of a graph that is not a forest is the number of edges of a smallest induced cycle in it.

The {\it line graph} $L(G)$ of a graph $G$ has vertex set $E(G)$ and there exists an edge between two vertices $e$ and $f$ in $L(G)$ if and only if $e$ and $f$ have a common end-vertex in $G$. The claw $K_{1,3}$ is the 4-vertex star. It is readily seen that every line graph is claw-free.
Recall that a graph is $H$-free if it does not contain $H$ as induced subgraph. For a set of graphs $\{H_1,\ldots,H_r\}$, we say that a graph $G$ is {\it $(H_1,\ldots,H_r)$-free} if $G$ is $H_i$-free for every $i\in \{1,\ldots,r\}$.

A {\it clique} is a set of pairwise adjacent vertices and an {\it independent set} is a set of pairwise non-adjacent vertices. A graph is {\it split} if its vertex set can be partitioned into two (possibly empty) sets, one of which is a clique and the other is an independent set. A graph is split if and only if it is $(C_4,C_5,P_4)$-free~\cite{FH77}. 
A graph is a {\it cograph} if it can be defined recursively as follows: any single vertex is a cograph, the disjoint union of two cographs is a cograph, and the join of two cographs $G_1,G_2$ is a cograph (the {\it join} adds all edges between the vertices of $G_1$ and $G_2$).
A graph is a cograph if and only if it is $P_4$-free~\cite{CLB81}.

A graph $G=(V,E)$ is {\it multipartite}, or more specifically, {\it $r$-partite} if $V$ can be partitioned into $r$ (possibly empty) sets $V_1,\ldots, V_r$, such that there is an edge between two vertices $u$ and $v$ if and only if $u\in V_i$ and $v\in V_j$ for some $i,j$ with $i\neq j$. If $r=2$, we also say that $G$ is {\it bipartite}. 
If there exist an edge between every vertex of $V_i$ and every vertex of $V_j$ for every $i\neq j$, then the multipartite graph $G$ is {\it complete}. 

The {\it complement} of a graph $G=(V,E)$ is the graph $\overline{G}=(V,\{uv\; |\; u,v\in V, u\neq v\; \mbox{and}\; uv\notin E\})$.  The complement of a bipartite graph is a {\it cobipartite} graph. 
A set $W \subseteq V$ is a \emph{dominating set} of a graph $G$ if every vertex of $V\setminus W$ has a neighbour in $W$, or equivalently, $N[W]$ (the closed neighbourhood of $W$) is equal to $V$. We say that $W$ is a \emph{connected dominating set} if $W$ is a dominating set and $G[W]$ is connected.

\section{The Proof of Theorem~\ref{t-main1}}\label{s-main1}

We consider {\sc $k$-Disjoint Connected Subgraphs} for fixed $k$. First, we show a polynomial-time algorithm on $H$-free graphs when $H\ssi sP_1+P_4$ for some fixed $s \geq 0$. Then, we prove the hardness result.

For the algorithm, we need the following lemma for $P_4$-free graphs, or equivalently, cographs.
This lemma is well known and follows immediately from the definition of a 
cograph: in the construction of a connected cograph $G$, the last operation must be a join, so there exists cographs $G_1$ and $G_2$, such that $G$ obtained from adding an edge between every vertex of  $G_1$ and every vertex of $G_2$. Hence, the spanning complete bipartite graph of $G$ has non-empty partition classes $V(G_1)$ and $V(G_2)$. 

\begin{lemma}\label{l-p4}
Every connected $P_4$-free graph on at least two vertices has a spanning complete bipartite subgraph.
\end{lemma}

\noindent
Two instances of a problem~$\Pi$ are {\it equivalent} when one of them is a yes-instance of $\Pi$ if and only if the other one is a yes-instance of $\Pi$. We note that if two adjacent vertices will always appear in the same set of every solution $(S_1,\ldots,S_k)$ for an instance $(G,Z_1,\ldots,Z_k)$, then we may contract the edge between them at the start of any algorithm. This takes linear time. Moreover, $H$-free graphs are readily seen (see e.g.~\cite{KP20}) to be closed under edge contraction if $H$ is a linear forest. Hence, we can make the following observation.

\begin{lemma}\label{l-contract}
For $k\geq 2$, from every instance of $(G,Z_1,\ldots,Z_k)$ of $k$-{\sc Disjoint Connected Subgraphs} we can obtain in polynomial time an equivalent instance $(G',Z_1',\ldots,Z_k')$ such that
every $Z_i'$ is an independent set. Moreover, if $G$ is $H$-free for some linear forest $H$, then $G'$ is also $H$-free.
\end{lemma}

\noindent
We can now prove the following lemma.

\begin{lemma}\label{l-sp1p4}
Let $H$ be a graph. If $H\ssi sP_1+P_4$, then for every  $k\geq 1$, {\sc $k$-Disjoint Connected Subgraphs} on $H$-free graphs is polynomial-time solvable.
\end{lemma}

\begin{proof}
Let $H\ssi sP_1+P_4$ for some $s\geq 0$.
Let $(G,Z_1,\ldots,Z_k)$ be an instance of {\sc $k$-Disjoint Connected Subgraphs}, where $G$ is an $H$-free graph.
By Lemma~\ref{l-contract}, we may assume without loss of generality that $G$ is connected and moreover that $Z_1,\ldots,Z_k$ are all independent sets.

We first analyze the structure of a solution $(S_1,\ldots,S_k)$ (if it exists). 
For $i\in \{1,\ldots,k\}$, we may assume that $S_i$ is inclusion-wise minimal, meaning there is no $S_i' \subset S_i$ that contains $Z_i$ and is connected.
Consider a graph $G[S_i]$. Either $G[S_i]$ is $P_4$-free or $G[S_i]$ contains an induced $rP_1+P_4$ for some $0\leq r\leq s-1$. 
We will now show that in both cases, $S_i$ is the (not necessarily disjoint) union of $Z_i$ and a connected dominating set of $G[S_i]$ of constant size.

First suppose that $G[S_i]$ is $P_4$-free. As $G[S_i]$ is connected and $Z_i$ is independent,
 we apply Lemma~\ref{l-p4} to find that $S_i\setminus Z_i$ contains a vertex~$u$ that is adjacent to every vertex of $Z_i$.
Hence, by minimality, $S_i = Z_i \cup \{u\}$ and $\{u\}$ is a connected dominating set of $G[S_i]$ of size~$1$.

Now suppose that $G[S_i]$ has an induced $rP_1+P_4$ for some $r\geq 0$, where we choose $r$ to be maximum. Note that $r\leq s-1$.
Let $W$ be the vertex set of the induced $rP_1+P_4$. Then, as $r$ is maximum, $W$ dominates $G[S_i]$. 
Note that $G[W]$ has $r+1 \leq s$ connected components.
Then, as $G[S_i]$ is connected and $W$ is a dominating set of $G[S_i]$ of size $r+4 \leq s+3$, it follows from folklore arguments
 (see e.g.~\cite[Prop.~6.3.24]{vanLeeuwen-thesis}) that $G[S_i]$ has a connected dominating set $W'$ of size at most $3s+1$. Moreover, by minimality, $S_i = Z_i \cup W'$.

Hence, in both cases we find that $S_i$ is the union of $Z_i$ and a connected dominating set of $G[S_i]$ of size at most $t = 3s+1$; note that $t$ is a constant, as $s$ is a constant.

Our algorithm now does as follows. We consider all options of choosing a connected dominating set of each $G[S_i]$, which from the above has size at most $t$. 
As soon as one of the guesses makes every $Z_i$ connected, we stop and return the solution.
The total number of options is $O(n^{tk})$, which is polynomial as $k$ and $t$ are fixed. Moreover, checking the connectivity condition can be done in polynomial time. Hence, the total running time of the algorithm is polynomial. \qed
\end{proof}

\noindent
The proof our next result is inspired by the aforementioned \NP-completeness result of~\cite{HPW09} for instances $(G,Z_1,Z_2)$ where $|Z_1|=2$ but $G$ is a general graph.

\begin{lemma}\label{l-line}
The  {\sc $2$-Disjoint Connected Subgraphs} problem is \NP-complete even on instances $(G,Z_1,Z_2)$ where $|Z_1|=2$ and $G$ is a line graph. \end{lemma}
\begin{proof}
Note that the problem is in \NP. We reduce from $3$-SAT. Let $\phi = \phi(x_1, \dots , x_n)$ be an instance of $3$-SAT with clauses $C_1, \dots, C_m$. We construct a corresponding graph $G=(V,E)$ as follows.
We start with two disjoint paths $P$ and $\bar{P}$ on vertices $p_i, x_i,q_i$ and $\bar{p}_i,\bar{x}_i,\bar{q}_i$, respectively, where $x_i, \bar{x}_i$ correspond to the positive and negative literals in $\phi$, respectively. To be more precise, we define:
\begin{equation*}
P= p_1,x_1,q_1, p_2,x_2,q_2, \dots, p_n, x_n, q_n,\; \mbox{and}\;\;
\overline{P}=\bar{p}_1, \bar{x}_1, \bar{q}_1, \dots, \bar{p}_n, \bar{x}_n, \bar{q}_n,
\end{equation*}
We add the two edges $e=p_1\bar{p}_1, \;\mbox{and}\; f=q_n\bar{q}_n$. 
For $i=1, \dots, n-1$, we also add
edges $q_i\bar{p}_{i+1}$ and $\bar{q}_ip_{i+1}$.
We now replace each $x_i$ by vertices $x_i^{j_1}, x_i^{j_2}, \dots x_i^{j_r}$, where $j_1,  \dots, j_r$ are the indices of the clauses $ C_j$ that contain $x_i$. That is, we replace
the subpath $p_i,x_i,q_i$ of $P$ by the path $p_i, x_i^{j_1}, x_i^{j_2}, \dots x_i^{j_r}, q_i$. We do the same path replacement operation on $\bar{P}$ with respect to every $\bar{x}_i$. 
Finally, we add every clause~$C_j$ as a vertex and add an edge between $C_j$ and $x_i^j$ if and only if $x_i\in C_j$, and between $C_j$ and $\bar{x}_i^j$ if and only if $\bar{x}_j\in C_j$. This completes the description of $G=(V,E)$. We refer to Figure~\ref{f-ex} for an illustration of our construction.

\begin{figure}
		\resizebox{5.5cm}{!} {
			\begin{minipage}{0.45\textwidth}
				\centering
				\[
				\xymatrix@R3pc @C0.75pc{ 
					&& {C_1}{\circ} \ar@{-}[d] \ar@{-}@/_-1.5pc/[drrrrr] \ar@{-}@/_4.25pc/[ddrrrrrrrrrr]\\
					&{p_1}{\bullet} \ar@{-}[r]  & {x^1_1} {\circ} \ar@{-}[r]& {\circ}  \ar@{-}[r] & {\circ} \ar@{-}[r] & {q_1}{\bullet} \ar@{-}[r] \ar@{-}[rd]  & {p_2}{\bullet}\ar@{-}[r] \ar@{-}[ld]&  {x^1_2}{\circ}  \ar@{-}[r] & {\circ} \ar@{-}[r] & {\circ} \ar@{-}[r] & {q_2}{\bullet} \ar@{-}[r] \ar@{-}[rd] & {p_3}{\bullet} \ar@{-}[r] \ar@{-}[ld] & {\circ} \ar@{-}[r] & {\circ} \ar@{-}[r]  & {q_3}{\bullet} \\
					&{\bar{p_1}}{\bullet} \ar@{-}[u]^*+{ e} \ar@{-}[r] & {\bar{x^1_1}} {\circ} \ar@{-}[r]& {\circ}  \ar@{-}[r] & {\circ} \ar@{-}[r] & {\bar{q_1}}{\bullet} \ar@{-}[r]  & {\bar{p_2}}{\bullet} \ar@{-}[r]&  {\circ}  \ar@{-}[r] & {\circ} \ar@{-}[r] & {\circ} \ar@{-}[r] & {\bar{q_2}}{\bullet} \ar@{-}[r] & {\bar{p_3}}{\bullet} \ar@{-}[r] & {\bar{x^1_3}}{\circ} \ar@{-}[r] & {\circ} \ar@{-}[r] &  {\bar{q_3}}{\bullet} \ar@{-}[u]^*+{ f} \\
				}
				\]
			\end{minipage}
		
		}
		\caption{The construction described with edges added for the clause $C_1=(x_1 \lor x_2 \lor \bar{x_3})$.}\label{f-ex}
\end{figure}
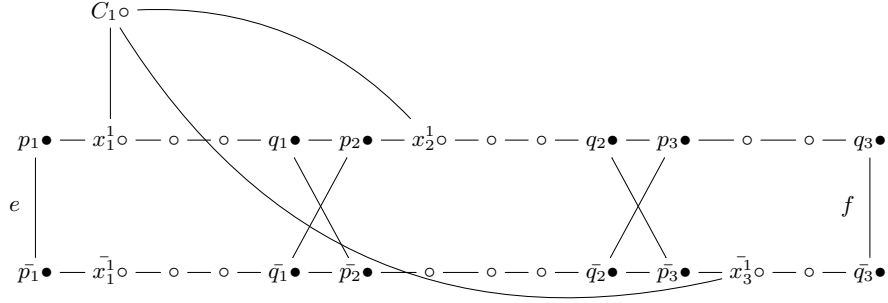	

We now focus on the line graph $L=L(G)$ of $G$.  Let $Z_1 = \{e, f\} \subseteq E=V(L)$ and let $Z_2$ consist of all vertices of $L$ that correspond to edges in $G$ that are incident to some $C_j$. Note that $Z_1$ and $Z_2$ are disjoint.
Moreover, each clause $C_j$ corresponds to a clique of size at most~$3$ in $L$, which we call the clause clique of $C_j$.
We claim that  $\phi$ is satisfiable if and only if the instance $(L,Z_1,Z_2)$ of {\sc $2$-Disjoint Connected Subgraphs} is a yes-instance.

First suppose that $\phi$ is satisfiable. Let $\tau$ be a satisfying truth assignment for~$\phi$.
In~$G$, we let 
$P^1$ denote the unique path whose first edge is $e$ and whose last edge is $f$ and that passes through all $x_i^j \in V$ if $x_i=0$ and through all $\bar{x}_i^j$ if
$x_i=1$. In~$L$ we let $S_1$ consist of all vertices of $L(P^1)$; note that $Z_1=\{e,f\}$ is contained in $S_1$ and that $S_1$ is connected.
We  let $P^2$ denote the ``complementary'' path in $G$ whose first edge is $e$ and whose last edge is $f$ but
that passes through all $x_i^j$ if and only if $P^1$ passes through all $\bar{x}_i^j$, and conversely ($i=1, \dots, n$).
In $L$, we put all vertices of $L(P^2)$, except $e$ and $f$, together with all vertices of $Z_2$ in~$S_2$. 
As $\tau$ satisfies $\phi$, some vertex of each clause clique is adjacent to a vertex of~$P^2$.
Hence, as $P^2$ is a path, $S_2$ is connected and we found a solution for $(L,Z_1,Z_2)$.

Now suppose that $(L,Z_1,Z_2)$ is a yes-instance of {\sc $2$-Disjoint Connected Subgraphs}. Then $V(L)$ can be partitioned into 
two vertex-disjoint connected sets $S_1$ and $S_2$ such that $Z_1\subseteq S_1$ and $Z_2\subseteq S_2$.
In particular, $L[S_1]$ contains a path~$P^1$ from $e$ to $f$. In fact, we may assume that $S_1=V(P^1)$, as we can move every other vertex of $S_1$ (if they exist) to $S_2$ without disconnecting $S_2$. 

Note that $P^1$ corresponds to a 
connected subgraph that contains the adjacent vertices $p_1$ and $\bar{p}_1$ as well as the adjacent vertices $q_n$ and $\bar{q}_n$.
Hence, we can modify $P^1$ into a path $Q$ in $G$ that starts in $p_1$ or $\bar{p}_1$ and that ends in $q_n$ or $\bar{q}_n$. Note that $Q$ contains no edge incident to a clause vertex $C_j$, as those edges correspond to vertices in $L$ that belong to $Z_2$. 
Hence, by construction, $Q$ ``moves from left to right'', that is, $Q$ cannot pass through both some $x_i^j$ and $\bar{x}_i^j$ (as then $Q$ needs to pass through either $x_i^j$ or $\bar{x}_i^j$ again implying that $Q$ is not a path). 

Moreover, if $Q$ passes through some $x_i^j$, then $Q$ must pass through
all vertices $x_i^{j_h}$. Similarly,  if $Q$ passes through some $\bar{x}_i^j$, then $Q$ must pass through
all vertices $\bar{x}_i^{j_h}$. As $Q$ connects the edges $p_1\bar{p}_1$ and $q_n\bar{q}_n$, we conclude that $Q$ must pass, for $i=1,\ldots,n$, through either every $x_i^{j_h}$ or through every $\bar{x}_i^{j_h}$. Thus we may define a truth assignment $\tau$ by setting 
\begin{equation*}
x_i=
\begin{cases}
1\;  \text{if}\; Q\; \text{passes through all}\; \bar{x}_i^j\\
0\; \text{if}\; Q\; \text{passes through all}\; x_i^j.
\end{cases}
\end{equation*} 
We claim that $\tau$ satisfies $\phi$. For contradiction, assume some clause $C_j$ is not satisfied. Then $Q$ passes through all its literals. However, then in $S_2$, the vertices of $Z_2$ that correspond to edges incident to $C_j$ are not connected to other vertices of $Z_2$, a contradiction. This completes the proof of the lemma.\qed
\end{proof}

\noindent
A straightforward modification of the reduction of Lemma~\ref{l-split} gives us Lemma~\ref{l-bipartite}. We can also obtain Lemma~\ref{l-bipartite} by subdividing the graph $G$ in the proof of Lemma~\ref{l-line} twice (to get a bipartite graph) or $p$ times (to get a graph of girth at least $p$).

\begin{lemma}[\cite{HPW09}]\label{l-split}
{\sc $2$-Disjoint Connected Subgraphs} is \NP-complete for split graphs, or equivalently, $(2P_2,C_4,C_5)$-free graphs.
\end{lemma}

\begin{lemma}\label{l-bipartite}
{\sc $2$-Disjoint Connected Subgraphs} is \NP-complete for  bipartite graphs and for graphs of girth at least $p$, for every integer~$p\geq 3$.
\end{lemma}

We are now ready to prove Theorem~\ref{t-main1}.

\medskip
\noindent
{\bf Theorem~\ref{t-main1} (restated)}
{\it Let $H$ be a graph. If $H\ssi sP_1+P_4$, then for every  $k\geq 1$, {\sc $k$-Disjoint Connected Subgraphs} on $H$-free graphs is polynomial-time solvable; otherwise even {\sc $2$-Disjoint Connected Subgraphs} is \NP-complete.}

\begin{proof}
If $H$ contains an induced cycle $C_s$ for some $s\geq 3$, then we apply Lemma~\ref{l-bipartite} by setting $p=s+1$. Now assume that $H$ contains no cycle, that is, $H$ is a forest. If $H$ has a vertex of degree at least~$3$, then $H$ is a superclass of the class of claw-free graphs, which in turn contains all line graphs. Hence, we can apply Lemma~\ref{l-line}. In the remaining case $H$ is a linear forest. If $H$ contains an induced $2P_2$, we apply Lemma~\ref{l-split}. Otherwise $H$ is an induced subgraph of $sP_1+P_4$ for some $s\geq 0$ and we apply Lemma~\ref{l-sp1p4}.\qed
\end{proof}

\section{The Proof of Theorem~\ref{t-main2}}\label{s-main2}

We first prove the following result, which generalizes the corresponding result of {\sc Disjoint Paths} for $P_4$-free graphs due to Gurski and Wanke~\cite{GW06}. We show that we can use the same modification to a matching problem in a bipartite graph.

\begin{lemma}\label{l-p4dcs}
{\sc Disjoint Connected Subgraphs} is polynomial-time solvable for $P_4$-free graphs.
\end{lemma}

\begin{proof}
For some integer $k\geq 2$, 
let $(G,Z_1,\ldots,Z_k)$ be an instance of {\sc Disjoint Connected Subgraphs} where $G$ is a $P_4$-free graph.
By Lemma~\ref{l-contract} we may assume that every $Z_i$ is an independent set. Now suppose that $(G,Z_1,\ldots,Z_k)$ has a solution $(S_1,\ldots,S_k)$. Then $G[S_i]$ is a connected $P_4$-free graph. Hence, by Lemma~\ref{l-p4}, $G[S_i]$ has a spanning complete bipartite graph on non-empty partition classes $A_i$ and $B_i$. As every $Z_i$ is an independent set, it follows that either $Z_i\subseteq A_i$ or $Z_i\subseteq B_i$. If $Z_i\subseteq A_i$, then every vertex of $B_i$ is adjacent to every vertex of $Z_i$. Similarly, if $Z_i\subseteq B_i$, then every vertex of $A_i$ is adjacent to every vertex of $Z_i$. We conclude that in every set $S_i$, there exists a vertex $y_i$ such that $Z_i\cup \{y_i\}$ is connected. 
	
The latter enables us to construct a bipartite graph $G^{\prime}=(X\cup Y, E^{\prime})$ where $X$ contains vertices $x_1,\ldots,x_k$ corresponding to the set $Z_1,\ldots,Z_k$ and $Y$ is the set of non-terminal vertices of $G$. We add an edge between $x_i\in X$ and $y\in Y$ if and only if $y$ is adjacent to every vertex of $Z_i$. Then $(G, Z_1 \dots Z_k)$ is a yes-instance of {\sc Disjoint Connected Subgraphs} if and only if $G^{\prime}$ contains a matching of size~$k$. It remains to observe that we can find a maximum matching in polynomial time, for example, by using the Hopcroft-Karp algorithm for bipartite graphs~\cite{HK73}.
\qed
\end{proof}

The first lemma of a series of four is obtained by a straightforward reduction from the {\sc Edge Disjoint Paths} problem (see, e.g.~\cite{GW06,HHLS15}), which was proven to be \NP-complete 
by Even, Itai and Shamir~\cite{EIS76}. 
The second lemma follows from the observation that an edge subdivision of the graph $G$ in an instance of {\sc Disjoint Paths} results in an equivalent instance of {\sc Disjoint Paths}; we apply this operation a sufficiently large number of times to obtain a graph of large girth. The third lemma is due to Heggernes et al.~\cite{HHLS15}. We modify their construction to prove the fourth lemma.

\begin{lemma}\label{l-line2}
{\sc Disjoint Paths} is \NP-complete for line graphs.
\end{lemma}

\begin{lemma}\label{l-girth}
For every $g\geq 3$, {\sc Disjoint Paths} is \NP-complete for graphs of girth at least~$g$.
\end{lemma}

\begin{lemma}[\cite{HHLS15}]\label{l-split2}
{\sc Disjoint Paths} is \NP-complete for split graphs, or equivalently, $(C_4,C_5,2P_2)$-free graphs.
\end{lemma}

\begin{lemma}\label{l-newhard}
{\sc Disjoint Paths} is \NP-complete for $(4P_1,P_1+P_4)$-free graphs.
\end{lemma}

\begin{proof}
We reduce from {\sc Disjoint Paths} on split graphs, which is \NP-complete by Lemma~\ref{l-split2}. By inspection of this result (see \cite[Theorem~3]{HHLS15}), we note that the instances $(G, \{(s_1,t_1),\ldots,(s_k,t_k)\})$ have the following property: the split graph~$G$ has a split decomposition $(C,I)$, where 
$C$ is a clique, $I$ an independent set, $C$ and $I$ are disjoint, and 
$C \cup I = V(G)$, 
such that $I = \{s_1,\ldots,s_k,t_1,\ldots,t_k\}$. 
Now let $G'$ be obtained from $G$ by, for each terminal $s_i$, adding edges to $s_j$ and $t_j$ for all $j \not= i$. Then consider the instance $(G', \{(s_1,t_1),\ldots,(s_k,t_k)\})$.

We note that $G'[C]$ is still a complete graph, while $G'[I]$ is a complete graph minus a matching. It is immediate that $G'$ is $4P_1$-free. Moreover, any induced subgraph $H$ of $G'$ that is isomorphic to $P_4$ must contain at least two vertices of $I$ and at least one vertex of $C$. If $H$ contains two vertices of $C$, then as $G'[C]$ is a clique, $H$ contains two non-adjacent vertices in $I$. Similarly, if $H$ contains one vertex of $C$ (and thus three vertices of $I$), then $H$ contains two non-adjacent vertices in $I$. Since $C$ is a clique in $G'$ and every (other) vertex of $I$ is adjacent in $G'$ to any pair of non-adjacent vertices of $I$, it follows that $G'$ is $P_1+P_4$-free as well.

We claim that $(G, \{(s_1,t_1),\ldots,(s_k,t_k)\})$ is a yes-instance if and only if $(G', \{(s_1,t_1),\ldots,(s_k,t_k)\})$ is a yes-instance. This is because the edges that were added to $G$ to obtain $G'$ are only between terminal vertices of different pairs. These edges cannot be used by any solution of {\sc Disjoint Paths} for $(G', \{(s_1,t_1),\ldots,(s_k,t_k)\})$, and thus the feasibility of the instance is not affected by the addition of these edges.
\qed\end{proof}

\noindent
We are now ready to prove Theorem~\ref{t-main2}.

\medskip
\noindent
{\bf Theorem~\ref{t-main2} (restated)}
{\it Let $H$ be a graph not in $\{3P_1,2P_1+P_2,P_1+P_3\}$. If $H\ssi P_4$, then {\sc Disjoint Connected Subgraphs}  
is polynomial-time solvable for $H$-free graphs; otherwise even {\sc Disjoint Paths} 
is \NP-complete.}

\begin{proof}
First suppose that $H$ contains a cycle $C_r$ for some $r\geq 3$. Then {\sc Disjoint Paths} is \NP-complete for the class of $H$-free graphs, as {\sc Disjoint Paths} is \NP-complete on the subclass consisting of graphs of girth $r+1$ by Lemma~\ref{l-girth}. Now suppose that $H$ contains no cycle, that is, $H$ is a forest. If $H$ contains a vertex of degree at least~$3$, then the class of $H$-free graphs contains the class of claw-free graphs, which in turn contains the class of line graphs. Hence, we can apply Lemma~\ref{l-line2}. It remains to consider the case where $H$ is a forest with no vertices of degree at least~$3$, that is, when $H$ is a linear forest.

If $H$ contains four connected components, then the class of $H$-free graphs contains the class of $4P_1$-free graphs, and we can use Lemma~\ref{l-newhard}. If $H$ contains an induced $P_5$ or two connected components that each have at least one edge, then $H$ contains the class of $2P_2$-free graphs, and we can use Lemma~\ref{l-split2}. If $H$ contains two connected components, one of which has at least four vertices, then $H$ contains the class of $(P_1+P_4)$-free graphs, and we can use Lemma~\ref{l-newhard} again. As  $H\notin \{3P_1,2P_1+P_2,P_1+P_3\}$, this means that in the remaining case $H$ is an induced subgraph of $P_4$. In that case even {\sc Disjoint Connected Subgraphs} is polynomial-time solvable on $H$-free graphs, due to Lemma~\ref{l-p4dcs}. \qed
\end{proof}

\section{Reducing the Number of Open Cases to Three}\label{s-reduce}

Theorem~\ref{t-main2} shows that we have the same three open cases for {\sc Disjoint Paths} and {\sc Disjoint Connected Subgraphs}, namely when $H\in \{3P_1,P_1+P_3,2P_1+P_2\}$. We show that instead of six open cases, we have in fact only three.

\begin{proposition}\label{o-1}
{\sc Disjoint Paths} and {\sc Disjoint Connected Subgraphs} are equivalent for $3P_1$-free graphs.
\end{proposition}

\begin{proof}
Every instance of {\sc Disjoint Paths} is an instance of {\sc Disjoint Connected Subgraphs}.
Let $(G,Z_1,\ldots,Z_k)$ be an instance of {\sc Disjoint Connected Subgraphs} where $G$ is a $3P_1$-free graph. By Lemma~\ref{l-contract} we may assume that each $Z_i$ is an independent set. Then, as $G$ is $3P_1$-free, each $Z_i$ has size at most~$2$. So we obtained an instance of {\sc Disjoint Paths}. \qed
\end{proof}

\begin{proposition}\label{o-2}
{\sc Disjoint Paths} on $(P_1+P_3)$-free graphs and {\sc Disjoint Connected Subgraphs} on $(P_1+P_3)$-free graphs are 
polynomially
equivalent to {\sc Disjoint Paths} on $3P_1$-free graphs.
\end{proposition}

\begin{proof}
We prove that we can solve {\sc Disjoint Connected Subgraphs} in polynomial time on $(P_1+P_3)$-free graphs if we have a polynomial-time algorithm for {\sc Disjoint Paths} on $3P_1$-free graphs.
Showing this suffices to prove the theorem, as {\sc Disjoint Paths} is a special case of {\sc Disjoint Connected Subgraphs} and $3P_1$-free graphs form a subclass of $(P_1+P_3)$-free graphs.

Let $(G,Z_1,\ldots,Z_k)$ be an instance of {\sc Disjoint Connected Subgraphs}, where $G$ is a $(P_1+P_3)$-free graph.
	 Olariu~\cite{Ol88} proved that every connected $\overline{P_1+P_3}$-free graph is either triangle-free or complete multipartite.
	Hence, the vertex set of $G$ can be partitioned into sets $D_1,\ldots,D_p$ for some $p\geq 1$ such that
	\begin{itemize}
		\item every $G[D_i]$ is $3P_1$-free or the disjoint union of complete graphs, and
		\item for every $i,j$ with $i\neq j$, every vertex of $D_i$ is adjacent to every vertex of $D_j$.
	\end{itemize}
Using this structural characterization, we first argue that we may assume that each $Z_i$ has size~$2$, making the problem an instance of {\sc Disjoint Paths}. Then we show that we can either solve the instance outright or can alter $G$ to be $3P_1$-free.

First, we argue about the size of each $Z_i$.
By Lemma~\ref{l-contract} we may assume that every $Z_i$ is an independent set and is thus contained in the same set $D_j$. If $G[D_j]$ is $3P_1$-free, then this implies that any $Z_i$ that is contained in $D_j$ has size~$2$. 
If $G[D_j]$ is a disjoint union of complete graphs, then each vertex of a $Z_i$ that is contained in $D_j$ belongs to a different connected component of $D_j$ and $Z_i \cup \{v\}$ is connected for every vertex $v \notin D_j$. As at least one vertex $v\notin D_j$ is needed to make such a set $Z_i$ connected, we may therefore assume that for a solution $(S_1,\ldots,S_k)$ (if it exists), $S_i= Z_i \cup \{v\}$ for some $v \notin D_j$. 
The latter implies that we may assume without loss of generality that every such $Z_i$ has size~$2$ as well.
	
If $p=1$, then each connected component of $G$ is $3P_1$-free, and we are done.
Hence, we assume that $p\geq 2$.	
In fact, since any two distinct sets $D_i$ and $D_j$ are complete to each other, the union of any two $3P_1$-free graphs induces a $3P_1$-free graph. Therefore we may assume without loss of generality that only $G[D_1]$ might be $3P_1$-free, whereas $G[D_2], \ldots, G[D_p]$ are disjoint unions of complete graphs.

Recall that $Z_i=\{s_i,t_i\}$ for every
$i\in \{1,\ldots,k\}$ 
and we search for a solution $(P^1,\ldots,P^k)$ where each $P^i$ is a path from $s_i$ to $t_i$. First suppose $s_i$ and $t_i$ belong to $D_1$. Then $P^i$ has length $2$ or $3$ and in the latter case, $V(P^i)\subseteq D_1$. 
Now suppose that  $s_i$ and $t_i$ belong to $D_h$ for some $h\in \{2,\ldots,k\}$. Then~$P^i$ has length exactly~$2$, and moreover, the middle (non-terminal) vertex of $P^i$ does not belong to $D_h$.

We will now check if there is a solution $(P^1,\ldots,P^k)$ such that every $P^i$ has length exactly~$2$. We call such a solution to be of {\it type~1}. 
In a solution of type~1, every $P^i=s_iut_i$ for some non-terminal vertex $u$ of $G$. 
If $s_i$ and $t_i$ belong to $D_h$ for some $h\in \{2,\ldots,p\}$, then $u\in D_j$ for some $j\neq i$. 
If $s_i$ and $t_i$ belong to $D_1$, then $u \in D_j$ for some $j \neq 1$ but also $u\in D_1$ is possible, namely when $u$ is adjacent to both $s_i$ and $t_i$.

Verifying the existence of a type~1 solution
is equivalent to finding a perfect matching in a bipartite graph $G^{\prime}= A \cup B$ that is defined as follows. The set~$A$ consists of one vertex $v_i$ for each pair $\{s_i,t_i\}$. The set $B$ consists of all non-terminal vertices $u$ of $G$. For $\{s_i,t_i\} \subseteq D_1$, there  exists an edge between $u$ and $v_i$ in $G^\prime$ if and only if in $G$ it holds that $u\in D_h$ for some $h\in \{2,\ldots,p\}$ or $u\in D_1$ and $u$ is adjacent to both $s_i$ and $t_i$. For $\{s_i,t_i\}\subseteq D_h$ with $h\in \{2,\ldots,p\}$, there exists an edge between $u$ and $v_i$ in $G^\prime$ if and only if in $G$ it holds that $u\in D_j$ for some $j\in \{1,\ldots,p\}$ with $h\neq j$. 
We can find a perfect matching in $G^\prime$ in polynomial time by using the Hopcroft-Karp algorithm for bipartite graphs~\cite{HK73}.

Suppose that we find that $(G,\{s_1,t_1\},\ldots,\{s_k,t_k\})$ has no solution of type~1. As a solution can be assumed to be of type~1 if $G[D_1]$ is the disjoint union of complete graphs, we find that $G[D_1]$ is not of this form. Hence, $G[D_1]$ is $3P_1$-free. Recall that $G[D_j]$ is the disjoint union of complete graphs for $2 \leq i \leq p$. It remains to check if there is a solution that is of {\it type~2} meaning a solution $(P^1,\ldots,P^k)$ in which at least one $P^i$, whose vertices all belong to $D_1$, has length~$3$. 

To find a  type~2 solution (if it exists) we
construct the following graph $G^*$. We let $V(G^*)=A_1\cup A_2 \cup B_1 \cup B_2$, where
\begin{itemize}
\item $A_1$ consists of all terminal vertices from $D_1$; 
\item $A_2$ consists of all non-terminal vertices from $D_1$;
\item $B_1$ consists of all terminal vertices from $D_2\cup \cdots \cup D_p$; and
\item $B_2$ consists of all non-terminal vertices from $D_2\cup \cdots \cup D_p$.
\end{itemize}
Note that $V(G^*)=V(G)$. To obtain $E(G^*)$ from $E(G)$ we add some edges (if they do not exist in $G$ already) and also delete some edges (if these existed in $G$):
\begin{enumerate}[(i)]
\item for each $\{s_i,t_i\}\subseteq B_1$, add all edges between $s_i$ and vertices of $B_2$, and delete any edges between $t_i$ and vertices of $B_2$;
\item add an edge between every two terminal vertices in $B_1$ that belong to different terminal pairs; and
\item add an edge between every two vertices of $B_2$. 
\end{enumerate}
We note that $G^*[D_1]$ is the same graph as $G[D_1]$ and thus $G^*[D_1]$ is $3P_1$-free. 
Moreover, $G^*[B_1\cup B_2]$ is $3P_1$-free by part (i) of the construction. 
Hence, as there exists an edge between every vertex of $A_1\cup A_2$ and every vertex of $B_1\cup B_2$ in $G$ and thus also in $G^*$, this means that $G^*$ is $3P_1$-free. It remains to prove that $(G,\{s_1,t_1\},\ldots,\{s_k,t_k\})$ and $(G^*,\{s_1,t_1\},\ldots,\{s_k,t_k\})$ are equivalent instances. 

First suppose that $(G,\{s_1,t_1\},\ldots,\{s_k,t_k\})$ has a solution $(P^1,\ldots,P^k)$. Assume that the number of paths of length~$3$ in this solution is minimum over all solutions for $(G,\{s_1,t_1\},\ldots,\{s_k,t_k\})$.
We note that $(P^1,\ldots,P^k)$ is a solution for $(G^*,\{s_1,t_1\},\ldots,\{s_k,t_k\})$ unless there exists some $P^i$ that contains an edge of $E(G)\setminus E(G^*)$. Suppose this is indeed the case. As $G^*[D_1]=G[D_1]$ and every edge between a vertex of $A_1\cup A_2$ and a vertex of $B_1\cup B_2$ also exists in $G^*$, we find that the paths connecting terminals from pairs in $D_1$ are paths in $G^*$. Hence, $s_i$ and~$t_i$ belong to $D_h$ for some $h\in \{2,\ldots,p\}$ and thus $P^i=s_iut_i$ where $u$ is a vertex of $D_j$ for some $j\in \{2,\ldots,p\}$ with $j\neq h$.

As we already found that  $(G,\{s_1,t_1\},\ldots,\{s_k,t_k\})$ has no type~1 solution, there is at least one $P^{i'}$ with length~$3$, so $P^{i'}=s_{i'}vv't_{i'}$ is in $G[D_1]$. However, we can now obtain another solution for $(G,\{s_1,t_1\},\ldots,\{s_k,t_k\})$ by changing $P^i$ into $s_ivt_i$ and $P^{i'}$ into $s_{i'}ut_{i'}$, a contradiction, as the number of paths of length~3 in $(P^1,\ldots,P^k)$ was minimum. We conclude that every $P^i$ only contains edges from $E(G)\cap E(G^*)$, and thus $(P^1,\ldots,P^k)$ is a solution for  $(G^*,\{s_1,t_1\},\ldots,\{s_k,t_k\})$.

Now suppose that  $(G^*,\{s_1,t_1\},\ldots,\{s_k,t_k\})$ has a solution $(P^1,\ldots,P^k)$. Consider a path $P^i$. First suppose that $s_i$ and $t_i$ both belong to $B_1$. Then we may assume without loss of generality that $P^i=s_iut_i$ for some $u\in A_2$. As $B_1$ only contains terminals from pairs in $D_2\cup \ldots \cup D_p$, the latter implies that $P^i$ is a path in $G$ as well. Now suppose that $s_i$ and $t_i$ both belong to $A_1$. Then we may assume without loss of generality that $P^i=s_iut_i$ for some non-terminal vertex of $V(G)=V(G^*)$ or $P^i=s_iuu't_i$ for two vertices $u,u'$ in $A_2\subseteq D_1$. Hence, $P^i$ is a path in $G$ as well.
We conclude that $(P^1,\ldots,P^k)$ is a solution for $(G,\{s_1,t_1\},\ldots,\{s_k,t_k\})$. This completes our proof.
\qed
\end{proof}

\section{Exact Algorithms}\label{s-exact}

In this section, 
we briefly mention exact algorithms. 
Using Held-Karp type dynamic programming techniques~\cite{Bellman,HeldKarp}, we can obtain exact algorithms for {\sc Disjoint Paths} and {\sc Disjoint Connected Subgraphs} running in time $O(2^nn^2m)$ and $O(3^nkm)$, respectively.

\begin{theorem}\label{t-exact1}
{\sc Disjoint Paths} can be solved in $O(2^nn^2k)$ time.
\end{theorem}
\begin{proof}
We devise a Held-Karp type~\cite{HeldKarp,Bellman} dynamic programming algorithm. Given a set $S \subseteq V(G)$, a vertex $v \in S$, and an integer $i \in\{1,\ldots,k\}$, let $D[S,v,i]$ be true if and only if $S$ can be partitioned into vertex-disjoint paths $P^1,\ldots,P^i$ such that $P^i$ starts in $s_i$ and ends in $v$ and $P^j$ is an $s_j$-$t_j$ path for each $j \in \{1,\ldots,i-1\}$.
Then we set $D[S,v,1]$ to true if and only if $S$ is equal to the vertex set of an $s_1$-$v$ path. The correctness of the base case is immediate from the definition.
Beyond the base case, we set $D[S,s_i,i] = D[S\setminus\{s_i\},t_{i-1}, i-1]$ and for all $v \not= s_i$, $D[S,v,i]$ is set to true if and only if there is a neighbour $w \in S$ of $v$ for which $D[S\setminus\{v\},w,i]$ is true.
Indeed, if $S$ can be partitioned into vertex-disjoint paths $P^1,\ldots,P^i$ such that $P^i$ starts in $s_i$ and ends in $v$ and $P^j$ is an $s_j$-$t_j$ path for each $j \in \{1,\ldots,i-1\}$, then 
\begin{itemize}
\item if $v=s_i$, then $P^i$ is a single-vertex path and thus $S\setminus\{s_i\}$ can be partitioned into vertex-disjoint paths $P^1,\ldots,P^{i-1}$ such that $P^j$ is an $s_j$-$t_j$ path for each $j \in \{1,\ldots,i-1\}$, and thus $D[S\setminus\{s_i\},t_{i-1},i-1]$ is true;
\item otherwise, let $w$ be the vertex preceding $v$ on $P^i$, and thus $S\setminus\{v\}$ can be partitioned into vertex-disjoint paths $P^1,\ldots,P^{i-1},Q^i$ such that $Q^i$ starts in $s_i$ and ends in $w$ ($Q^i$ is the part of $P^i$ from $s_i$ to $w$) and $P^j$ is an $s_j$-$t_j$ path for each $j \in \{1,\ldots,i-1\}$, and thus $D[S\setminus\{v\},w,i]$ is true.
\end{itemize}
Conversely, if $v=s_i$ and $D[S\setminus\{s_i\},t_{i-1},i-1]$ is true, then $S\setminus\{s_i\}$ can be partitioned into vertex-disjoint paths $P^1,\ldots,P^{i-1}$ such that $P^j$ is an $s_j$-$t_j$ path for each $j \in \{1,\ldots,i-1\}$, and thus $S$ can be partitioned into vertex-disjoint paths $P^1,\ldots,P^i$ such that $P^i$ starts and ends in $s_i$ and $P^j$ is an $s_j$-$t_j$ path for each $j \in \{1,\ldots,i-1\}$. Hence, $D[S,s_i,i]$ is true. If $v\not=s_i$ and there is a neighbour $w \in S$ of $v$ for which $D[S\setminus\{v\},w,i]$ is true, meaning that $S\setminus\{v\}$ can be partitioned into vertex-disjoint paths $P^1,\ldots,P^i$ such that $P^i$ starts in $s_i$ and ends in $w$ and $P^j$ is an $s_j$-$t_j$ path for each $j \in \{1,\ldots,i-1\}$, then $S$ can be partitioned into vertex-disjoint paths $P^1,\ldots,P^{i-1},Q^i$ such that $Q^i$ starts in $s_i$, follows $P^i$ and ends in $v$, and $P^j$ is an $s_j$-$t_j$ path for each $j \in \{1,\ldots,i-1\}$. Hence, $D[S,v,i]$ is true.

Finally, the given instance of {\sc Disjoint Paths} is a yes-instance if and only if there is a set $S \subseteq V(G)$ for which $D[S,t_k,k]$ is true. The correctness follows by definition.

It is immediate that the running time of the algorithm is $O(2^nn^2k)$, as there are $2^nnk$ table entries that each require at most $O(n)$ time to fill.
\qed\end{proof}

\begin{theorem}\label{t-exact2}
{\sc Disjoint Connected Subgraphs} can be solved in $O(3^nkm)$ time.
\end{theorem}

\begin{proof}
We propose a similar, but slightly more crude algorithm as the one before. Given a set $S \subseteq V(G)$ and an integer $i \in\{1,\ldots,k\}$, let $D[S,i]$ be true if and only if $S$ can be partitioned into vertex-disjoint set $S_1,\ldots,S_i$ such that $S_j$ is connected and $Z_j \subseteq S_j$ for each
$j \in \{1,\ldots,i\}$. We set $D[S,1]$ to true if and only if $Z_1 \subseteq S$ and $S$ is connected. Beyond the base case, we set $D[S,i]$ to true if and only if there is a set $S' \subset S$ for which $Z_i \subseteq S'$, $S'$ is connected, and $D[S\setminus S',i-1]$ is true. Finally, the given instance of {\sc Disjoint Connected Subgraphs} is a yes-instance if and only if there is a set $S \subseteq V(G)$ for which $D[S,k]$ is true. The proof of correctness is similar (but simpler) to the proof of Theorem~\ref{t-exact1}.

It is immediate that the running time is $O(3^nkm)$. Each table entry $D[S,i]$ requires $O(2^{|S|}m)$ time to fill. Hence, the running time to fill all table entries where $S$ has size $\ell$ is $k\binom{n}{\ell}2^\ell m$.
This means that the total running time is $\sum_{\ell=0}^n\binom{n}{\ell}2^\ell mk=O(3^nkm)$, where the latter equality follows from the Binomial Theorem.
\qed\end{proof}

\section{Conclusions}\label{s-con}
We first gave a dichotomy for {\sc Disjoint $k$-Connected Subgraphs} in Theorem~\ref{t-main1}: for every $k$, the problem is 
polynomial-time 
solvable on $H$-free graphs if $H\ssi sP_1+P_4$ for some $s\geq 0$ and otherwise it is \NP-complete even for $k=2$.
Two vertices~$u$ and~$v$ are a {\it $P_4$-suitable pair} if $(G-\{u,v\},N(u),N(v))$ is a yes-instance of {\sc $2$-Disjoint Connected Subgraphs}. Recall that a graph $G$ can be contracted to $P_4$ if and only if $G$ has a $P_4$-suitable pair. Deciding if a pair~$\{u,v\}$ is a suitable pair is polynomial-time solvable for $H$-free graphs if $H$ is an induced subgraph of  $P_2+P_4$, $P_1+P_2+P_3$, $P_1+P_5$ or $sP_1+P_4$ for some $s\geq 0$; otherwise it is \NP-complete~\cite{KP20}.
Hence, we conclude from our new result that the presence of the two vertices $u$ and $v$ that are connected to the sets $Z_1=N(u)$ and $Z_2=N(v)$, respectively, yield exactly three additional polynomial-time solvable cases.

We also classified, in Theorem~\ref{t-main2}, the complexity of {\sc Disjoint Paths} and {\sc Disjoint Connected Subgraphs} for $H$-free graphs. Due to Propositions~\ref{o-1} and~\ref{o-2}, there are three non-equivalent open cases left and we ask the following:

\medskip
\noindent
{\bf Open Problem 1.} {\it Determine the computational complexity of {\sc Disjoint Paths} on $H$-free graph for $H\in \{3P_1,2P_1+P_2\}$ and the computational complexity of {\sc Disjoint Connected Subgraphs} on $H$-free graphs for $H=2P_1+P_2$.}

\medskip
\noindent
The three open cases seem challenging. We were able to prove the following positive result for a subclass of $3P_1$-free graphs, namely cobipartite graphs, or equivalently, $(3P_1,C_5,\overline{C_7}, \overline{C_9},\ldots)$-free graphs.

\begin{theorem}\label{t-cob}
		{\sc Disjoint Paths} is polynomial-time solvable for cobipartite graphs.
	\end{theorem}
	
	\begin{proof}
Let $G=(A \cup B, E)$, with cliques $A$ and $B$, be the given cobipartite graph.
If $s_i$ and $t_i$ are adjacent in $G$, then use the direct edge between them as the path $P^i$. We can then reduce the instance by removing $s_i$ and $t_i$. We now assume the instance has thus been reduced and (by abuse of notation) all terminal pairs are nonadjacent in $G$.

We now construct a bipartite graph $G^{\prime}$ by removing each edge within the cliques $A$ and $B$ as well as any edge $s_it_j$ both of whose endpoints are terminals. We then obtain a new graph $G^{\prime \prime}$ by deleting each terminal vertex and adding for each terminal pair $(s_i, t_i)$, a new vertex $x_i$ whose neighbourhood is the union of the neighbourhoods of $s_i$ and $t_i$ in $G^{\prime}$. We claim that $G$ contains the required $k$ disjoint paths $P^1 \dots P^k$ if and only if $G^{\prime \prime}$ contains a matching of size at least $k$. We can check the latter in polynomial time by using the Hopcroft-Karp algorithm for bipartite graphs~\cite{HK73}.
		
		We first assume that $G$ contains the disjoint paths $P^1 \dots P^k$. Note that, since $G$ is $3P_1$-free, we may assume each path has length at most $3$. A matching $M$ of size $k$ is obtained as follows. For each $i=1 \dots k$, if $P^i$ has length $2$ we add the edge $x_iv_i$ to $M$ where $v_i$ is the interior vertex of $P^i$. If $P^i$ has length $3$ then we add its interior edge $u_iv_i$ to $M$.
		
		 Next assume $G^{\prime \prime}$ contains a matching $M$ of size $k$. For each edge of $M$ which includes a vertex $x_i$ corresponding to a terminal pair $(s_i, t_i)$ we set $P^i$ to be $s_iv_it_i$ where $v_i$ is the vertex matched to $x_i$. Note that any edge $uv$ in $G$ which contains no terminal vertex and has one endpoint in each of $A$ and $B$ lies on a path of length $3$ between any two terminal vertices. Therefore, for each $i$ such that the vertex $x_i$ is not matched in $M$, we can choose a distinct edge $u_iv_i$ in $M$ to obtain the path $s_iu_iv_it_i$ in $G$. \qed
	\end{proof}
	
Finally, in Section~\ref{s-exact} we obtained exact algorithms for {\sc Disjoint Paths} and {\sc Disjoint Connected Subgraphs} running in time $O(2^nn^2m)$ and $O(3^nkm)$, respectively.
Faster exact algorithms are known for {\sc $k$-Disjoint Connected Subgraphs} for $k=2$ and $k=3$~\cite{CPPW14,TV13,AFLST20}, but we are unaware if there exist faster algorithms for general graphs.

\medskip
\noindent
{\bf Open Problem 2.} {\it Is there an exact algorithm for {\sc Disjoint Paths} or {\sc Disjoint Connected Subgraphs} on general graphs where the exponential factor is $(2-\epsilon)^n$ or $(3-\epsilon)^n$, respectively, for some $\epsilon > 0$?}

\bibliographystyle{abbrv}

\begin{thebibliography}{10}

\bibitem{AFLST20}
A.~Agrawal, F.~V. Fomin, D.~Lokshtanov, S.~Saurabh, and P.~Tale.
\newblock Path contraction faster than $2^n$.
\newblock {\em S{I}{A}{M} Journal on Discrete Mathematics}, 34:1302--1325,
  2020.

\bibitem{Bellman}
R.~Bellman.
\newblock Dynamic programming treatment of the travelling salesman problem.
\newblock {\em Journal of the ACM}, 9:61--63, 1962.

\bibitem{CLB81}
D.~G. Corneil, H.~Lerchs, and L.~S. Burlingham.
\newblock Complement reducible graphs.
\newblock {\em Discrete Applied Mathematics}, 3:163--174, 1981.

\bibitem{CPPW14}
M.~Cygan, M.~Pilipczuk, M.~Pilipczuk, and J.~O. Wojtaszczyk.
\newblock Solving the 2-disjoint connected subgraphs problem faster than $2^n$.
\newblock {\em Algorithmica}, 70:195--207, 2014.

\bibitem{EIS76}
S.~Even, A.~Itai, and A.~Shamir.
\newblock On the complexity of timetable and multicommodity flow problems.
\newblock {\em S{I}{A}{M} Journal on Computing}, 5:691--703, 1976.

\bibitem{FH77}
S.~F{\"o}ldes and P.~L. Hammer.
\newblock Split graphs.
\newblock {\em Congressus Numerantium}, XIX:311--315, 1977.

\bibitem{GKLS12}
C.~Gray, F.~Kammer, M.~L{\"{o}}ffler, and R.~I. Silveira.
\newblock Removing local extrema from imprecise terrains.
\newblock {\em Computational Geometry: Theory and Applications}, 45:334--349,
  2012.

\bibitem{GW06}
F.~Gurski and E.~Wanke.
\newblock Vertex disjoint paths on clique-width bounded graphs.
\newblock {\em Theoretical Computer Science}, 359:188--199, 2006.

\bibitem{HHLS15}
P.~Heggernes, P.~van~'t Hof, E.~J. van Leeuwen, and R.~Saei.
\newblock Finding disjoint paths in split graphs.
\newblock {\em Theory of Computing Systems}, 57:140--159, 2015.

\bibitem{HeldKarp}
M.~Held and R.~M. Karp.
\newblock A dynamic programming approach to sequencing problems.
\newblock {\em Journal of the Society for Industrial and Applied Mathematics},
  10:196--210, 1962.

\bibitem{HK73}
J.~E. Hopcroft and R.~M. Karp.
\newblock An $n^{5/2}$ algorithm for maximum matchings in bipartite graphs.
\newblock {\em {SIAM} Journal on Computing}, 2:225--231, 1973.

\bibitem{Ka75}
R.~M. Karp.
\newblock On the complexity of combinatorial problems.
\newblock {\em Networks}, 5:45--68, 1975.

\bibitem{KKR12}
K.~Kawarabayashi, Y.~Kobayashi, and B.~A. Reed.
\newblock The disjoint paths problem in quadratic time.
\newblock {\em Journal of Combinatorial Theory, Series {B}}, 102:424--435,
  2012.

\bibitem{KP20}
W.~Kern and D.~Paulusma.
\newblock {Contracting to a longest path in ${H}$-free graphs}.
\newblock {\em Proc. ISAAC 2020, LIPIcs}, 181:22:1--22:18, 2020.

\bibitem{NS96}
S.~Natarajan and A.~P. Sprague.
\newblock Disjoint paths in circular arc graphs.
\newblock {\em Nordic Journal of Computing}, 3:256--270, 1996.

\bibitem{Ol88}
S.~Olariu.
\newblock Paw-free graphs.
\newblock {\em Information Processing Letters}, 28:53--54, 1988.

\bibitem{PR11}
D.~Paulusma and J.~M.~M. van Rooij.
\newblock On partitioning a graph into two connected subgraphs.
\newblock {\em Theoretical Computer Science}, 412(48):6761--6769, 2011.

\bibitem{RS95}
N.~Robertson and P.~D. Seymour.
\newblock Graph minors .{X}{I}{I}{I}. the disjoint paths problem.
\newblock {\em Journal of Combinatorial Theory, Series {B}}, 63:65--110, 1995.

\bibitem{Sh80}
Y.~Shiloach.
\newblock A polynomial solution to the undirected two paths problem.
\newblock {\em Journal of the {ACM}}, 27:445--456, 1980.

\bibitem{TV13}
J.~A. Telle and Y.~Villanger.
\newblock Connecting terminals and 2-disjoint connected subgraphs.
\newblock {\em Proc. WG 2013, LNCS}, 8165:418--428, 2013.

\bibitem{vanLeeuwen-thesis}
E.~J. van Leeuwen.
\newblock {\em Optimization and Approximation on Systems of Geometric Objects}.
\newblock University of Amsterdam, 2009.

\bibitem{HPW09}
P.~van~'t Hof, D.~Paulusma, and G.~J. Woeginger.
\newblock Partitioning graphs into connected parts.
\newblock {\em Theoretical Computer Science}, 410:4834--4843, 2009.

\end{thebibliography}

\end{document}